\documentclass[11pt]{article}
\usepackage[english]{babel}
\usepackage{amssymb, amsmath, amsthm} 
\usepackage[a4paper, centering]{geometry}
\usepackage{graphicx}
\usepackage{caption}
\usepackage{subfigure,color}
\usepackage{hyperref}
\usepackage[font={small,up}]{caption}

\numberwithin{equation}{section} 

\usepackage{pgfplots}
\usepackage{bm}

\usepackage{amsthm}
\makeatletter
\def\th@plain{%
  \thm@notefont{}
  \itshape 
}
\def\th@definition{%
  \thm@notefont{}
  \normalfont 
}
\makeatother
\newtheorem{thm}{Theorem}[section]
\newtheorem{prop}[thm]{Proposition}

\newtheorem{cor}[thm]{Corollary}

\theoremstyle{definition}
\newtheorem{defn}[thm]{Definition}

\newtheorem{oss}[thm]{Remark}

\newcommand{\Z}{\mathbb{Z}}

\newcommand{\N}{\mathbb{N}}
\newcommand{\R}{\mathbb{R}}

\renewcommand{\epsilon}{\varepsilon}

\usepackage{enumitem}
\setlist[description]{nosep}

\begin{document}

\title{Chirality transitions in frustrated ferromagnetic spin chains: a link with the gradient theory of phase transitions}

\author{{\scshape Giovanni Scilla}\\
Department of Mathematics and Applications ``R. Caccioppoli''\\ University of Naples ``Federico II''\\
Via Cintia, Monte S. Angelo - 80126 Naples \\
(ITALY)
 \\
\\
{\scshape Valerio Vallocchia}\\
Department of Mathematics \\ University of Rome ``Tor Vergata''\\
Via della Ricerca Scientifica 1, 00133 Rome \\
(ITALY)}
\date{}

\maketitle

\begin{abstract}
We study chirality transitions in frustrated ferromagnetic spin chains, in view of a possible connection with the theory of Liquid Crystals. A variational approach to the study of these systems has been recently proposed by Cicalese and Solombrino, focusing close to the helimagnet/ferromagnet transition point corresponding to the critical value of the frustration parameter $\alpha=4$. We reformulate this problem for any $\alpha\geq0$ in the framework of surface energies in nonconvex discrete systems with nearest neighbours ferromagnetic and next-to-nearest neighbours antiferromagnetic interactions and we link it to the gradient theory of phase transitions, by showing a uniform equivalence by $\Gamma\hbox{-}$convergence on $[0,4]$ with Modica-Mortola type functionals.\\ 
{\bf Keywords:} {$\Gamma\hbox{-}$convergence, Equivalence, Frustrated lattice systems, Chirality transitions, Modica-Mortola}
\end{abstract}

\section{Introduction}
\label{intro}

The phenomenon of frustration arises from the competition between different interactions, in a continuous or discrete physical system, that
favor incompatible ground states. It occurs, for instance, in the liquid-crystalline phases of
\emph{chiral} molecules: a chiral molecule cannot be superimposed on its mirror image through any proper rotation or translation. The main effect of chirality is that chiral molecules do not align themselves parallel to their neghbors but tend to form a characteristic angle with them (see, e.g., \cite{BK,KS,Dierking}).

Edge-sharing chains of cuprates, instead, provide an example of frustrated lattice systems, where the frustration results from the competition between ferromagnetic (F) nearest-neighbour (NN) and antiferromagnetic (AF) next-nearest-neighbour (NNN) interactions (see,  e.g.,~\cite{D}). 

In this paper we study the asymptotic properties of a one-dimensional frustrated spin system at zero temperature {via $\Gamma$\hbox{-}convergence (see \cite{GCB} and \cite{DalMaso})}, focusing also on the variational equivalence with problems in gradient theory of phase transitions (see, e.g.,~\cite{GCB,BH} for a simple introduction to the topic). {Our contribution has been inspired by the recent results about the variational discrete-to-continuum analysis of such systems provided by Cicalese and Solombrino~\cite{CS} in the vicinity of the so called ``helimagnet/ferromagnet transition point'', exhibiting at a suitable scale different scenarios not detected by a first-order $\Gamma$\hbox{-}limit. Indeed, the $\Gamma$\hbox{-}convergence approach provides a rigorous way of deriving a continuum limit for discrete systems as the number of interacting particles is increasing. However, the $\Gamma$\hbox{-}limit does not always capture the main features of the discrete model and in some cases more refined approximations are needed (see, e.g.,~\cite{BC,CS,Trusk,Bra13}). This motivated the derivation of the \emph{uniformly $\Gamma$\hbox{-}equivalent theories}, introduced by Braides and Truskinovsky~\cite{BT} for a wide class of discrete systems and developed, e.g., in the framework of fracture mechanics, by Scardia, Schl\"omerkemper and Zanini~\cite{SSZ} for one-dimensional chains of atoms with Lennard-Jones interactions between nearest-neighbours.} Our paper can also be seen as a first step in the analysis of chirality transitions in more complicated physical systems like as chiral liquid crystals. A discrete-to-continuum analysis via $\Gamma\hbox{-}$convergence of some problems in liquid crystals has been recently treated, e.g., by Braides, Cicalese and Solombrino~\cite{BCS}, but this promising research field is still largely unexplored.

We consider the so-called F-AF spin chain model, where the state of the system is described by an $S^1$-valued spin variable $u=(u^i)$ parameterized over the points of the set $\frac{1}{n}\Z\cap[0,1]$, $n\in\N$. The energy of a given state of the system is 
\begin{equation}
E_n^\alpha(u)=-\alpha\sum_{i=0}^{n-1}(u^i,u^{i+1})+\sum_{i=0}^{n-1}(u^i,u^{i+2})-nm_\alpha,
\label{introenergy0}
\end{equation}
with \emph{periodic} boundary conditions $(u^0,u^1)=(u^{n},u^{n+1})$, where $\alpha\geq0$ is the \emph{frustration parameter}, $(\cdot,\cdot)$ denotes the scalar product between vectors in $\R^2${and $m_\alpha$ are constants depending on $\alpha$ (see (\ref{minimalpha}) for the precise definition).

The first term of the energy (\ref{introenergy0}) is ferromagnetic and favors the alignment of NN spins, while the second,
being antiferromagnetic, frustrates it as it favors antipodal NNN spins. Consequently, the frustration of the system depends on the parameter $\alpha$. In order to characterize the ground states of this system and their dependence on the value of $\alpha$, we first associate to each pair of nearest neighbours $u^i,u^{i+1}$ the corresponding oriented central angle $\theta^i\in[-\pi,\pi)$. Then, by the periodicity assumption, we may reread the energies in terms of this scalar variable as
\begin{equation}
E_n^\alpha(\theta)=-\frac{\alpha}{2}\sum_{i=0}^{n-1}\Bigl(\cos\theta^i+\cos\theta^{i+1}\Bigr)+\sum_{i=0}^{n-1}\cos(\theta^i+\theta^{i+1})-nm_\alpha,
\label{introenergy2}
\end{equation}
{and follow the approach by Braides and Cicalese~\cite{BC} for lattice systems of the form (\ref{introenergy2}). Indeed, by ``minimizing out'' for each fixed $i$ the nearest neighbours interactions, we are led to the definition of the \emph{effective potential} $W_\alpha$ (equation (\ref{Wpotential})) such that 
\begin{equation}
E_n^\alpha(\theta)\geq\displaystyle\sum_{i=0}^{n-1}W_\alpha(\theta^i),
\label{introinequality}
\end{equation}
where $W_\alpha$ is convex with minimum at $\theta=\theta_\alpha=0$ if $\alpha\geq4$, while it is a double-well potential with wells at $\theta=\pm\theta_\alpha$ if $0\leq\alpha\leq4$ (see Fig.~\ref{potentialfig}). Since the inequality in (\ref{introinequality}) is strict if $\theta^i\neq\theta^{i+1}$ or $\theta^i\neq\pm\theta_\alpha$, we deduce that if $\alpha\geq4$ the nearest neighbours prefer to stay aligned (ferromagnetic order); if $0\leq\alpha\leq4$, instead, the minimal configurations of $E_n^\alpha$ are $\theta^i=\theta^{i+1}\in\{\pm\theta_\alpha\}$; that is, the angle between pairs of nearest neighbours $u^i,u^{i+1}$ and $u^{i+1},u^{i+2}$ is constant and depending on the particular value of $\alpha$ (helimagnetic order). The two possible choices for $\theta_\alpha$ (a degeneracy known in literature as \emph{chirality symmetry}) correspond to either clockwise or counterclockwise spin rotations, or, equivalently, to a positive or a negative chirality (see Fig.~\ref{figchir}).

\begin{figure}[htp]
\center
\includegraphics[scale=.4]{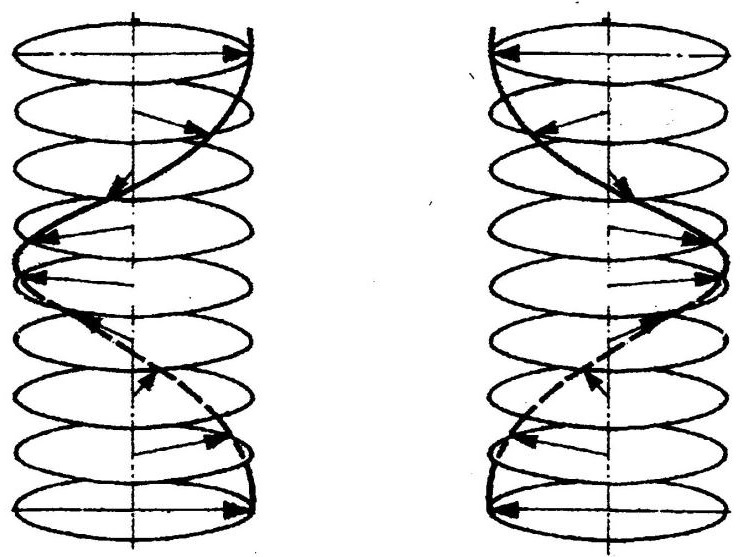}
\caption{A schematic representation of the ground states of the spin system for $0\leq \alpha < 4$ for clockwise (on the left) and counterclockwise (on the right) chirality (picture taken from \cite{Dierking}).} 
\label{figchir}
\end{figure}

The asymptotic behaviour of energies $E_n^\alpha$ as $n\to\infty$ and for fixed $\alpha$ (Theorem~\ref{gammalim}) reflects such different regimes for the ground states. If $\alpha\geq4$ the limit is trivially finite (and equal to zero) only on the constant function $\theta\equiv0$, while if $0\leq\alpha<4$ it is finite on functions with bounded variation taking only the two values $\{\pm\theta_\alpha\}$ and it counts the number of chirality transitions. More precisely,
\begin{equation*}
\displaystyle\Gamma\hbox{-}{\lim_{n\to+\infty}}E^\alpha_n(\theta)=C_\alpha\#(S(\theta)),
\end{equation*}
where $S(\theta)$ is the jump set of function $\theta$ and $C_\alpha=C(\alpha)$ is the cost of each chirality transition. The value $C_\alpha$ (see Section~\ref{defcrease}) represents the energy of an interface which is obtained by means of a `discrete optimal-profile problem' connecting the two constant (minimal) states $\pm\theta_\alpha$. It is continuous as a function of $\alpha$ on the interval $[0,4)$ (as shown by Proposition~\ref{continuity}) and can be defined to be equal to 0 for $\alpha\geq4$. Moreover, $C_\alpha\to0$ as $\alpha\to4$ and (compare with \cite{DK} and Remark~\ref{osservazione})
\begin{equation}
C_\alpha\sim\frac{\sqrt{2}}{3}(4-\alpha)^{3/2},\, \text{as }\alpha\to4^-.
\label{asymptreg}
\end{equation}
In a recent paper \cite{CS}, Cicalese and Solombrino investigated the asymptotic behaviour of this system
close to the ferromagnet/helimagnet transition point; that is, they found the correct scaling (heuristically suggested by (\ref{asymptreg})) to detect the symmetry breaking and to compute the asymptotic behaviour of the scaled energy describing this phenomenon as $\alpha$ is close to 4. They let the parameter $\alpha$ depend on $n$ and be close to 4 from below; i.e., they rewrite energies (\ref{introenergy2}) in terms of $4-\alpha_n$, with $4-\alpha_n\to0$ as $n\to\infty$. 

We state their result in a slight different form, useful for the sequel. More precisely, we prove in Theorem~\ref{cicasol} that an analogous result can be obtained if we choose as order parameter the ``flat'' angular variable
\begin{equation*}
v=\frac{\theta}{\theta_\alpha},
\end{equation*}
which is equivalent to the variable considered in~\cite{CS} in the regime of small angles. We compute the $\Gamma\hbox{-}$limit $F^0$ as $n\to\infty$, $\alpha=\alpha_n\to4$ with respect to the strong $L^1$-topology of the scaled energies
\begin{equation}
F^{\alpha_n}_n(v):=\frac{E_n^{\alpha_n}(v)}{\mu_{\alpha_n}}=\frac{8E_n^{\alpha_n}(v)}{\sqrt{2}{(4-\alpha_n)^{3/2}}},
\label{riscalate}
\end{equation}
and show that, within this scaling, several regimes are possible depending on the value
\begin{equation*}
l:=\lim_n \frac{\sqrt{2}}{4n(4-\alpha_n)^{1/2}}.
\end{equation*}
Namely, if $l=0$ then $F^0(v)=\frac{8}{3}\#(S(v)),\,v\in BV(I,\{\pm1\})$, if $l=+\infty$ then $F^0$ is finite (and equal to zero) only on constant functions, while in the intermediate case $l\in(0,+\infty)$ we get
\begin{equation*}
F^0(v)=\displaystyle\frac{1}{l}\int_{I}\Bigl(v^2(t)-1\Bigr)^2\,dt + {l}\int_{I}(\dot{v}(t))^2\,dt,\,v\in W^{1,2}_{|per|}(I),
\end{equation*}
where $I=(0,1)$, $BV(I,\{\pm1\})$ is the space of functions of bounded variation defined on $I$ and taking the values $\{\pm1\}$, and $W^{1,2}_{|per|}(I)=\{v\in W^{1,2}(I):\, |v(0)|=|v(1)|\}$.

Motivated by the particular form of this result and in the spirit of Braides and Truskinovsky\cite{BT}, with Theorem~\ref{equivalence} we find a variational link between such energies (seen as a `parametrized' family of functionals) and the gradient theory of phase transitions, in the framework of the \emph{equivalence} by $\Gamma\hbox{-}$convergence. Roughly speaking, two families of functionals are equivalent by $\Gamma\hbox{-}$convergence if they have the same $\Gamma\hbox{-}$limit (see Definition~\ref{defequiv} and the subsequent ones for the rigorous definitions useful in this framework). More precisely, we show the \emph{uniform} equivalence by $\Gamma\hbox{-}$convergence on $[0,4]$ of the energies $F^{\alpha}_n(v)$ defined in (\ref{riscalate}) with the ``Modica-Mortola type'' functionals given by
\begin{equation*}
G_n^{\alpha}(v)=
\displaystyle\frac{1}{\mu_\alpha}\Bigl[{\lambda_{n,\alpha}}\int_{I}\Bigl(v^2-1\Bigr)^2\,dt + \frac{M_\alpha^2}{\lambda_{n,\alpha}}\int_{I}(\dot{v})^2\,dt\Bigr],\,v\in W^{1,2}_{|per|}(I),
\end{equation*}
where $\lambda_{n,\alpha}=2n\theta_\alpha^4$ and $M_\alpha=3C_\alpha/8$.

The value $\alpha_0=4$ is a \emph{singular point}, since the $\Gamma\hbox{-}$limit of $G_n^{\alpha}$ will depend on choice of the particular sequence $\alpha_n\to\alpha_0^-=4^-$. Each $\alpha_0\in[0,4)$, instead, is a \emph{regular point}; i.e., it is not singular. As a consequence of Theorem~\ref{equivalence}, we deduce (see Corollary~\ref{corollario}) the uniform equivalence of the energies $E^{\alpha}_n(\theta)$ for $\alpha\in[0,4)$ with the family
\begin{equation*}
H_n^{\alpha}(\theta)=
\displaystyle\frac{\lambda_{n,\alpha}}{\theta^4_{\alpha}}\int_{I}\Bigl(\theta^2(t)-{\theta^2_{\alpha}}\Bigr)^2\,dt + \frac{M_\alpha^2}{\lambda_{n,\alpha}\theta_{\alpha}^2}\int_{I}(\dot{\theta}(t))^2\,dt,\,\theta\in W^{1,2}_{|per|}(I),
\end{equation*}
whose potentials $\mathcal{W}_\alpha(\theta):=(\theta^2-{\theta^2_{\alpha}})^2$ have the wells located at the minimal angles $\theta=\pm\theta_\alpha$.

As a final remark, we would like to observe that our result can be useful also to analyze more general problems of interest for the applied community. For instance, a natural extension would be the case of $S^2$-valued spins, that has been recently investigated by Cicalese, Ruf and Solombrino~\cite{CRS} in the vicinity of the transition point. In that paper, the authors modify the energies penalizing the distance of the $S^2$ field from a finite number of copies of $S^1$ and prove the emergence of non-trivial chirality transitions. However, even in the case of values in $S^1$, the Villain Helical $XY$-model studied there could be attacked with our approach, at least in the regime of ``strong'' ferromagnetic interaction considered therein by the authors.

\section{Setting of the problem}
Preliminarily, we fix some notation that will be used throughout. We denote by $I=(0,1)$ and by $\lambda_n=\frac{1}{n},n\in\N$ a positive parameter.
Given $x\in\R$, we denote by $\lfloor x\rfloor$ the integer part of $x$. The symbol $S^1$ stands for the standard unit sphere of $\R^2$. Given a vector $v\in\R^2$ with components $v_1$ and $v_2$ with respect to the canonical basis of $\R^2$, we will use the notation $v=(v_1|v_2)$. Given two vectors $v,w\in\R^2$ we will denote by $(v,w)$ their scalar product. Here and in the following, $\mathcal{U}_n(I)$ will be the space of the functions $w:\lambda_n\Z\cap[0,1]\to S^1$, $\Theta_n(I)$ the space of the functions $\varphi:\lambda_n\Z\cap[0,1]\to [-\frac{\pi}{2},\frac{\pi}{2}]$ and we use the notation $w^i=w(i\lambda_n)$, $\varphi^i=\varphi(i\lambda_n)$; $\bar{\mathcal{U}}_n(I)$ will denote the subspace of those $w\in\mathcal{U}_n(I)$ satisfying the following \emph{periodic boundary condition}
\begin{equation}
(w^1,w^0)=(w^{n+1}, w^{n}).
\label{periodic}
\end{equation}
Analogously, $\bar{\Theta}_n(I)$ will denote the subspace of those $\varphi\in\Theta_n(I)$ such that $\varphi^0=\varphi^n$.

We will identify each lattice function $w\in\bar{\mathcal{U}}_n(I)$ with its piecewise-constant interpolation belonging to the class
\begin{equation*}
\mathcal{C}_n(I)=\{w:\R\to S^1:\,w(t)=w(\lambda_ni) \text{ if }t\in(i,i+1)\lambda_n,\,i\in\{0,1,\dots,n-1\}\},
\end{equation*}
while the symbol $\mathcal{D}_n(I)$ will denote the analogous space for functions $\varphi\in\bar{\Theta}_n(I)$.

Given a pair of vectors $v=(v_1|v_2), w=(w_1|w_2)\in S^1$, we define the function $\chi[v,w]: S^1\times S^1\rightarrow \{\pm1\}$  as
\begin{equation}
\chi[v,w]=\text{sign}(v_1w_2-v_2w_1),
\label{determinant}
\end{equation}
with the convention that $\text{sign}(0)=-1$, and the corresponding oriented central angle $\theta\in[-\pi,\pi)$ by
\begin{equation}
\theta=\chi[v,w]\arccos((v, w)).
\label{orangle}
\end{equation}
The positivity of the determinant in (\ref{determinant}) represents the counterclockwise ordering of the vectors $v$ and $w$.

\subsection{The model energies $E_n^\alpha$}

We consider the energy of a given state $u$ of the F-AF spin chain model, defined as
\begin{equation}
E_n^\alpha(u)=P_n^\alpha(u)-nm_\alpha=-\alpha\sum_{i=0}^{n-1}(u^i,u^{i+1})+\sum_{i=0}^{n-1}(u^i,u^{i+2})-nm_\alpha,
\label{energy}
\end{equation}
where $u\in\mathcal{C}_n(I)$, $\alpha\geq0$ and (see \cite[Proposition~3.2]{CS})
\begin{equation}
m_\alpha= \frac{1}{n}\min_{u\in L^\infty(I,S^1)}P_n^\alpha(u)=
\begin{cases}
-\Bigl(\frac{\alpha^2}{8}+1\Bigr) & \text{if }\alpha\in[0,4],\\
-\alpha+1 & \text{if }\alpha\in[4,+\infty).
\end{cases}
\label{minimalpha}
\end{equation}

First we note that, thanks to the periodicity assumption (\ref{periodic}), we can write the energies (\ref{energy}) equivalently in the form
\begin{equation}
E_n^\alpha(u)=-\frac{\alpha}{2}\sum_{i=0}^{n-1}\Bigl[(u^i,u^{i+1})+(u^{i+1},u^{i+2})\Bigr]+\sum_{i=0}^{n-1}(u^i,u^{i+2})-nm_\alpha.
\label{energy1}
\end{equation}
Now we associate to each pair of neighbouring spins $u^i,u^{i+1}$ the corresponding oriented central angle $\theta^i$ defined as in (\ref{orangle}), and taking $\theta^i$ as (scalar) order parameter, the energies (\ref{energy1}) can be rewritten as
\begin{equation}
E_n^\alpha(\theta)=-\frac{\alpha}{2}\sum_{i=0}^{n-1}\Bigl(\cos\theta^i+\cos\theta^{i+1}\Bigr)+\sum_{i=0}^{n-1}\cos(\theta^i+\theta^{i+1})-nm_\alpha,
\label{energy2}
\end{equation}
where $\theta\in\mathcal{D}_n(I)$.

\subsection{Ground states of $E_n^\alpha$}\label{groundst}
In this section, we focus on the ground states of the energies $E_n^\alpha$. We will show the emergence of \emph{chiral} ground states for $\alpha\in[0,4]$ by means of a double-minimization technique introduced by Braides and Cicalese in \cite{BC} for lattice systems of the form (\ref{energy2}). Following their approach, for each $i=0,1,\dots,n-1$ we fix the next-to-nearest neighbour interactions $\theta^i+\theta^{i+1}=2\theta$ and solve the minimum problem
\begin{equation}
\min_{\theta^i\in[-\pi,\pi)}\Bigl\{-\frac{\alpha}{2}\Bigl[\cos\theta^i+\cos(2\theta-\theta^i)\Bigr]+\cos2\theta-m_\alpha\Bigr\}.
\label{minprob}
\end{equation}
By a direct computation, we find that the unique minimizers in (\ref{minprob}) for $\alpha\neq 0 $ are   $\theta^i=\theta^{i+1}=\theta$ if $\theta\in(-\pi/2,\pi/2)$, and  $\theta^i=\theta^{i+1}=\theta-\pi$ if {$|\theta| \in(\pi/2,\pi)$}, while for $\alpha =0$ we have $\theta^i=\theta^{i+1}=\pm\pi/2$. The following picture shows that, up to a reparametrization, $\theta$ and $\theta-\pi$ actually represent the same minimizer.\\

\begin{minipage}[b]{0.39\linewidth}

\includegraphics[width=1\textwidth]{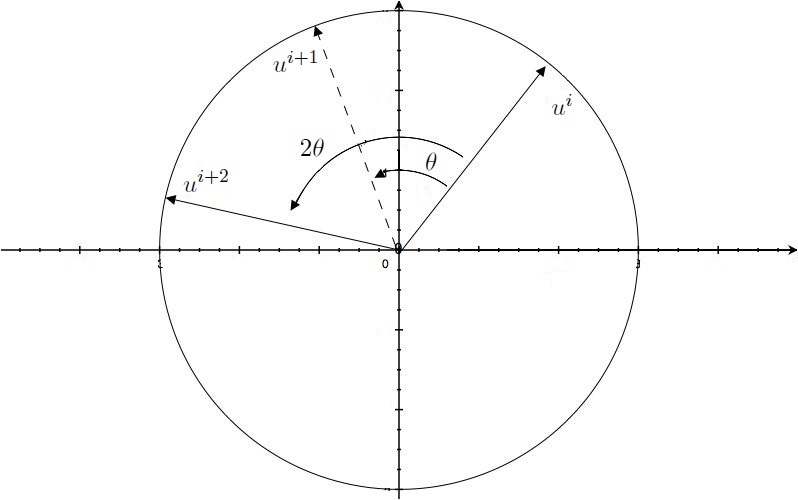} 
\begin{center}(a) The angle between NN is $\theta$
\end{center}
\end{minipage}
\begin{minipage}[b]{0.39\linewidth}

\includegraphics[width=1\textwidth]{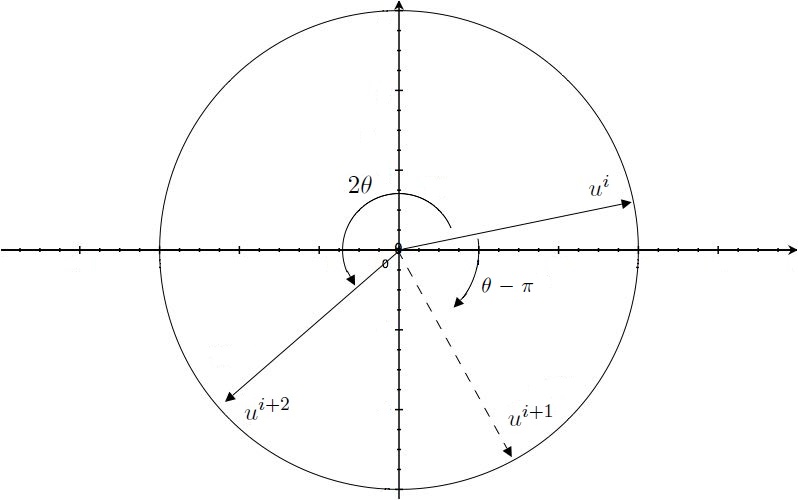} 
\begin{center} (b) The angle between NN is $\theta-\pi$
\end{center}
\end{minipage}
\\
\\
Without loosing generality we will assume up to the end that $\theta^i=\theta^{i+1}=\theta\in J$, $J:=[-\pi/2,\pi/2]$ and correspondingly we define the \emph{effective potential} as
\begin{equation}
W_\alpha(\theta)=\cos2\theta-\alpha\cos\theta-m_\alpha.
\label{Wpotential}
\end{equation}

The potential $W_\alpha$ is thus obtained by integrating out the effect of nearest-neighbour interactions
optimizing over atomic-scale oscillations, and its properties strongly depend on the value $\alpha$. 
Indeed, if $0\leq\alpha<4$ then $W_\alpha$ is a ``double-well'' potential, while if $\alpha\geq4$ the potential is convex (see Fig.~\ref{potentialfig}). Moreover,
\begin{equation}
\arg\min W_\alpha(\theta)=
\begin{cases}
\{\pm\theta_\alpha\}:=\{\pm\arccos(\frac{\alpha}{4})\}, & \text{if }\alpha\in[0,4],\\
\{0\}, & \text{if }\alpha\in[4,+\infty),
\end{cases}
\end{equation}
We note that by the definition of $W_\alpha$ and (\ref{minprob}) we get
\begin{equation}
E_n^\alpha(\theta)\geq\sum_{i=0}^{n-1}W_\alpha(\theta^i),
\label{inequality}
\end{equation}
the inequality being strict if $\theta^i\neq\theta^{i+1}$ or $\theta^i\neq\pm\theta_\alpha$. In particular, $E_n^\alpha(\theta)\geq0$.

\begin{figure}[htb]
\includegraphics[scale=0.33]{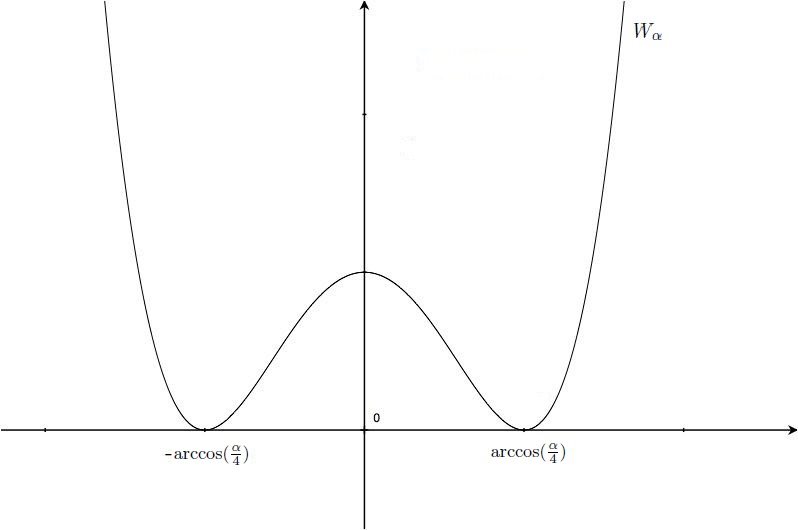}\qquad
\includegraphics[scale=0.33]{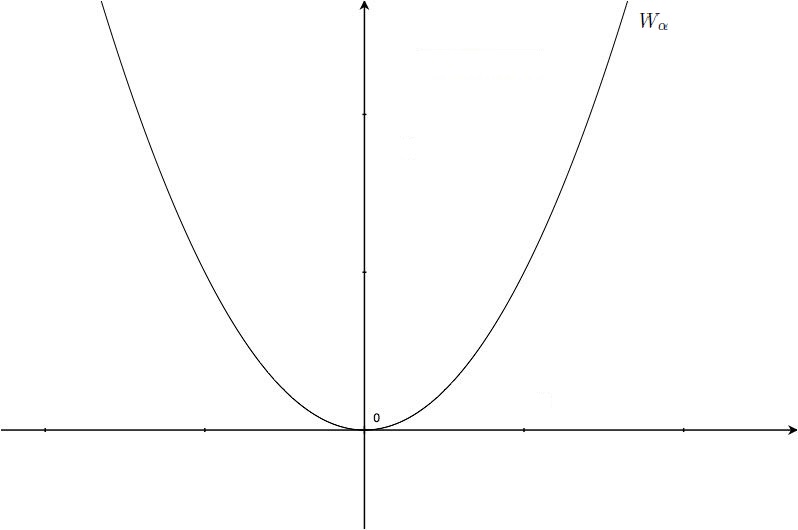}
\caption{The potential $W_\alpha$ for $0\leq\alpha<4$ (on the left) and for $\alpha\geq4$ (on the right).}
\label{potentialfig}
\end{figure}

Thus, the minimization procedure leading to the definition of $W_\alpha$ (and then to inequality (\ref{inequality})) allows us to deduce some information about the ground states of the energies $E_n^\alpha$ from the properties of this potential. More precisely, if $\alpha\leq4$ the minimal configurations of $E_n^\alpha$ are $\theta^i=\theta^{i+1}\in\{\pm\theta_\alpha\}$; that is, the angle between pairs of nearest neighbours $u^i,u^{i+1}$ and $u^{i+1},u^{i+2}$ is constant and depending on the particular value of $\alpha$. If $\alpha\geq4$, instead, the nearest neighbours prefer to stay aligned ($-\theta_\alpha=+\theta_\alpha=0$). \\

Let be $\theta\in\mathcal{D}_n(I)$. We may regard the energies $E_n^\alpha$  as defined on a subset of $L^\infty(I,J)$ and consider their extension on $L^\infty(I,J)$. With a slight abuse of notation, we set $E_n^\alpha: L^\infty(I,J)\to[0,+\infty]$ as
\begin{equation}
E_n^\alpha(\theta)=
\begin{cases}
\displaystyle-\frac{\alpha}{2}\sum_{i=0}^{n-1}(\cos\theta^i+\cos\theta^{i+1})+\sum_{i=0}^{n-1}\cos(\theta^i+\theta^{i+1})-nm_\alpha, & \text{if }\theta\in\mathcal{D}_n(I),\\
+\infty, & \text{otherwise}.
\end{cases}
\label{exten}
\end{equation}

\section{Limit behaviour of $E_n^\alpha$ with fixed $\alpha$}

Our first result is the explicit computation of the $\Gamma\hbox{-}$limit, as $n\to\infty$, of the energies $E_n^\alpha$ with fixed $\alpha\in[0,+\infty)$. As we will show with Theorem~\ref{gammalim}, the asymptotic behaviour of the energies $E_n^\alpha$ reflects the different regimes for the ground states outlined in Section~\ref{groundst}. Indeed, the limit is non-trivial only in the helimagnetic regime ($0\leq\alpha<4$), representing the energy the system spends on the scale 1 for a finite number of chirality transitions from $-\theta_\alpha$ to $\theta_\alpha$.

\subsection{Crease transition energies}\label{defcrease}

The cost $C_\alpha$ of each chirality transition can be characterized as the energy of an interface which is obtained by means of a `discrete optimal-profile problem' connecting the two constant (minimal) states $\pm\theta_\alpha$. 

Let $\alpha\in[0,4)$. According to \cite[Section~2.2]{BC}, we define the \emph{crease transition energy} between $-\theta_\alpha$ and $\theta_\alpha$ as 
\begin{equation}
\begin{split}
C_\alpha :=C(-\theta_\alpha,\theta_\alpha)&\\
=\inf_{N\in\N}\min\Bigl\{\sum_{i=-\infty}^{+\infty}&\Bigl[\cos(\theta^{i}+\theta^{i+1})-\frac{\alpha}{2}(\cos\theta^{i}+\cos\theta^{i+1})-m_\alpha\Bigr]:\\
&\,\theta:\Z\to[-\pi/2,\pi/2],\,\theta^i=\text{sign}(i)\theta_\alpha,\, \text{if $|i|\geq N$}\Bigr\}.
\end{split}
\label{crease}
\end{equation}
We note that the infinite sums in (\ref{crease}) are well defined, since they involve only non negative terms and, actually, for fixed $N$ they are finite sums, since the summands are 0 for $i\geq N$ and $i\leq-N-1$. Moreover, it follows by the definition a useful symmetry property of the crease energy; that is, 
\begin{equation}
C(-\theta_\alpha,\theta_\alpha)=C(\theta_\alpha,-\theta_\alpha).
\label{symmetry}
\end{equation}

Now we prove that the optimal test function in (\ref{crease}) is constantly equal to $\pm\theta_\alpha$ only for $N\to+\infty$, thus relaxing the boundary condition as a condition at infinity in the definition of $C_\alpha$. We notice that an analogous property of crease energies has been showed by Braides and Solci in~\cite{BOx} for a one-dimensional system of Lennard-Jones nearest and next-to-nearest neighbour interactions.

\begin{prop}\label{propcrease}
The infimum in {\rm(\ref{crease})} is obtained for $N\to\infty$; that is,
\begin{equation}
\begin{split}
C_\alpha=\inf\Bigl\{\sum_{i=-\infty}^{+\infty} &\Bigl[\cos(\theta^{i}+\theta^{i+1})-\frac{\alpha}{2}(\cos\theta^{i}+\cos\theta^{i+1})-m_\alpha\Bigr]:\\
&\theta:\Z\to[-\pi/2,\pi/2], \lim_{i\to\pm\infty}\emph{sign}(i)\theta^i=\theta_\alpha\Bigr\}.
\end{split}
\label{crease2}
\end{equation}
Moreover, $C_\alpha>0$.
\end{prop}

\begin{proof}
Let $\theta^i$ be a test function for the problem (\ref{crease2}) and denote by $\widetilde{C}_\alpha$ the infimum in (\ref{crease2}). With fixed $\eta>0$, let $N_\eta$ be such that $|\theta^i-\text{sign}(i)\theta_\alpha|<\eta$ for $|i|\geq N_\eta$, and define
\begin{equation*}
\theta^i_\eta=
\begin{cases}
\theta^i, & \text{if $|i|\leq N_\eta$}\\
\text{sign}(i)\theta_\alpha, & \text{if $|i|>N_\eta$}.
\end{cases}
\end{equation*}
We then have
\begin{equation*}
\begin{split}
&\sum_{i=-\infty}^{+\infty}\left[-\frac{\alpha}{2}(\cos\theta^{i}_\eta+\cos\theta^{i+1}_\eta)+\cos(\theta^{i}_\eta+\theta^{i+1}_\eta)-m_\alpha\right]\\
=&\sum_{i=-N_\eta-1}^{N_\eta}\left[-\frac{\alpha}{2}(\cos\theta^{i}_\eta+\cos\theta^{i+1}_\eta)+\cos(\theta^{i}_\eta+\theta^{i+1}_\eta)-m_\alpha\right]\\
=&\sum_{i=-N_\eta}^{N_\eta-1}\left[-\frac{\alpha}{2}(\cos\theta^{i}+\cos\theta^{i+1})+\cos(\theta^{i}+\theta^{i+1})-m_\alpha\right]\\
&-\frac{\alpha}{2}(\cos\theta^{N_\eta}+\cos\theta_{\alpha})+\cos(\theta^{N_\eta}+\theta_\alpha)-m_\alpha\\
&-\frac{\alpha}{2}(\cos\theta_{\alpha}+\cos\theta^{-N_\eta})+\cos(\theta^{-N_\eta}-\theta_\alpha)-m_\alpha\\
\leq &\sum_{i=-\infty}^{+\infty}\left[-\frac{\alpha}{2}(\cos\theta^{i}+\cos\theta^{i+1})+\cos(\theta^{i}+\theta^{i+1})-m_\alpha\right]+2\omega(\eta)
\end{split}
\end{equation*}
where
\begin{equation}
\omega(\eta):=\max\left\{-\frac{\alpha}{2}(\cos\theta+\cos\theta_\alpha)+\cos(\theta+\theta_\alpha)-m_\alpha:\,|\theta-\theta_\alpha|\leq\eta\right\}
\label{function}
\end{equation}
is infinitesimal as $\eta\to0$. This shows that the value $C_\alpha$ defined in (\ref{crease}) is less or equal than $\widetilde{C}_\alpha$. Then we are done, the converse inequality being trivial since any test function for problem (\ref{crease}) is a test function for problem (\ref{crease2}). The estimate $C_\alpha>0$ easily follows from (\ref{inequality}) and the fact that $\pm\theta_\alpha$ are the unique minimizers of $W_{\alpha}$.
\end{proof}

\begin{oss}
In the ferromagnetic regime $\alpha\geq4$, we may define $C_\alpha=0$ consistently with (\ref{crease}), where now $m_\alpha=-\alpha+1$. Indeed, being $\theta_\alpha=-\theta_\alpha=0$, we can choose $\theta\equiv0$ as a test function in (\ref{crease})  thus obtaining the estimate $C_\alpha\leq0$.
\end{oss}

It will be useful in the sequel the following continuity property of $C_\alpha$ with respect to the frustration parameter $\alpha$.

\begin{prop}[Continuity]
The crease energy $C_\alpha$ defined as before is continuous in $[0,4)$; i.e., for any $\bar{\alpha}\in[0,4)$ and any sequence $\alpha_j$ such that $0\leq\alpha_j<4$, $\alpha_j\to\bar{\alpha}$ it results $C_{\alpha_j}\to C_{\bar{\alpha}}$.
\label{continuity}
\end{prop}

\begin{proof}
Let us fix $\eta>0$ and let $\alpha,\alpha'\in[0,4)$ be such that if $|\alpha-\alpha'|<\delta(\eta)$ for a suitable $\delta(\eta)>0$, then $|\theta_\alpha-\theta_{\alpha'}|<\eta/2$.\\
 From the definition of $C_{\alpha'}$ as in (\ref{crease2}), there exists a function $\theta:\Z\to[-\pi/2,\pi/2]$ such that $\displaystyle\sum_{i\in\Z}\mathcal{E}^{i,\alpha'}(\theta)<C_{\alpha'}+\eta$, where we have set
\begin{equation*}
\mathcal{E}^{i,\alpha'}(\theta):=-\frac{\alpha'}{2}\left(\cos\theta^i+\cos\theta^{i+1}\right)+\cos(\theta^i+\theta^{i+1})+\frac{(\alpha')^2}{8}+1,
\end{equation*}
and $\displaystyle\lim_{i\to\pm\infty}\text{sign}(i)\theta^i=\theta_{\alpha'}$.
This means that there exist two indices $h_1(\eta), h_2(\eta)\in \mathbb{N}$ such that
$|\theta^i-\theta_{\alpha'}|<\eta/2$ for every $i>h_2(\eta)$ and $|\theta^i-(-\theta_{\alpha'})|<\eta/2$ for every $i<-h_1(\eta)$. 

Setting $\bar{h}=\bar{h}(\eta):=\max\{h_1,h_2\}$ and $K_{\bar{h}}:=\{i\in \mathbb{Z}: |i|\leq\bar{h}\}$, we observe that for every $i\not\in K_{\bar{h}}$ it also holds that $|\text{sign}(i)\theta^i-\theta_{\alpha}|<\eta$.

Now we modify $\theta$ in order to obtain a test function for the problem defining $C_\alpha$ by setting
\begin{equation}
\tilde{\theta}^i=
\begin{cases}
\theta^i, & \text{ if } i\in K_{\bar{h}}\\
\text{sign}(i)\theta_{\alpha}, & \text{otherwise.}
\end{cases}
\end{equation}

We then have
\begin{equation}
\begin{split}
C_\alpha &\leq\displaystyle\sum_{i\in\Z}\mathcal{E}^{i,\alpha}(\tilde{\theta})
=\displaystyle\sum_{|i|<\bar{h}}\mathcal{E}^{i,\alpha}({\theta})+\mathcal{E}^{-\bar{h},\alpha}(\tilde{\theta})
+\mathcal{E}^{\bar{h},\alpha}(\tilde{\theta})\\
 &\leq\displaystyle\sum_{i\in\Z}\mathcal{E}^{i,\alpha'}(\theta) + 2|\alpha-\alpha'|\#K_{\bar{h}}+\left[\mathcal{E}^{-\bar{h},\alpha}(\tilde{\theta})-\mathcal{E}^{-\bar{h},\alpha'}(\theta)\right]\\
&+\left[\mathcal{E}^{\bar{h},\alpha}(\tilde{\theta})-\mathcal{E}^{\bar{h},\alpha'}(\theta)\right],
\end{split}
\label{stimaprincipale}
\end{equation}
where in the second inequality we used the estimate
\begin{align*}
\displaystyle\sum_{|i|<\bar{h}}\left|\mathcal{E}^{i,\alpha}({\theta})-\mathcal{E}^{i,\alpha'}({\theta})\right|\leq  2|\alpha-\alpha'|\#K_{\bar{h}}.
\end{align*}

Each of the last two terms in \eqref{stimaprincipale} can be estimated in the same way, so we make an explicit computation only for the latter. We have

\begin{align*}
&\left|\mathcal{E}^{\bar{h},\alpha}(\tilde{\theta})-\mathcal{E}^{\bar{h},\alpha'}(\theta)\right|\leq\left|(\cos(\theta^{\bar{h}}+\theta_{\alpha})-\cos(\theta^{\bar{h}}+\theta^{\bar{h}+1})\right|+\left|\frac{\alpha^2-(\alpha')^2}{8}\right|+\\
&\left|\frac{\alpha}{2}(\cos\theta^{\bar{h}}+\cos\theta_{\alpha})-\frac{\alpha'}{2}(\cos\theta^{\bar{h}}+\cos\theta^{\bar{h}+1})\right|\leq \eta+2|\alpha-\alpha'|.
\end{align*}
Collecting all the previous estimates and inserting them into \eqref{stimaprincipale} we obtain
\begin{align*}
C_\alpha&\leq C_{\alpha'} +\eta + 2|\alpha-\alpha'|\#K_{\bar{h}} + \eta+2|\alpha-\alpha'|\\
&\leq C_{\alpha'}+2\eta+2|\alpha-\alpha'|(1+\#K_{\bar{h}}).
\end{align*}
Choosing now $\gamma\geq 4\eta$ and $\alpha$ and $\alpha'$ such that
\begin{align*}
|\alpha-\alpha'|\leq \min\left\{\delta(\eta),\frac{\gamma}{4(1+\#K_{\bar{h}})}\right\}=:\sigma(\gamma,\eta),
\end{align*}
we finally obtain $C_\alpha\leq C_{\alpha'}+\gamma$.

If we change the role of $\alpha$ and $\alpha'$, we get an analogous estimate for $C_{\alpha'}$. Hence, we conclude that for every $\gamma\geq4\eta$, there exists $\sigma(\gamma,\eta)>0$ such that if $|\alpha-\alpha'|<\sigma(\gamma,\eta)$ then $|C_\alpha-C_{\alpha'}|\leq \gamma$. Since the choice of $\eta$ was arbitrary, the assertion immediately follows.
\end{proof}

\subsection{Compactness and $\Gamma$\hbox{-}convergence results}

The following compactness result states that sequences $\theta_n$ with equibounded energy $E_n^\alpha$ converge to a limit function $\theta$ which has a finite number of jumps and takes the values $\{\pm\theta_\alpha\}$ almost everywhere if $0\leq\alpha<4$, while if $\alpha\geq4$ the limit function is identically $0$. 

\begin{prop}[Compactness]\label{compact2}
Let $E_n^\alpha: L^\infty(I,J)\to[0,+\infty]$ be the energies defined by {\rm(\ref{exten})}. If $\{\theta_n\}$ is a sequence of functions such that
\begin{equation}
\sup_n E_n^{\alpha}(\theta_n)<+\infty,
\label{boundedness2}
\end{equation}
then we have two cases:\\
{\rm(i)} if $0\leq\alpha<4$ there exists a set $S\subset(0,1)$ with $\#S<+\infty$ such that, up to subsequences, $\theta_n$ converges to $\theta$ in $L^{1}_{loc}((0,1)\backslash S)$, where $\theta$ is a piecewise constant function and $\theta(0+)=\theta(1+)$. Moreover, $\theta(t)\in\{\pm\theta_\alpha\}$ for a.e. $t\in(0,1)$ and $S(\theta)\subseteq S$;\\
{\rm(ii)} if $\alpha\geq4$ then the limit function $\theta$ is identically 0.
\end{prop}

\begin{proof}
(i) We first note that $-\alpha\cos\theta\geq |p_\alpha\theta|-\alpha-1$ for $\theta\in[-\pi/2,\pi/2]$ and a constant $p_\alpha$ depending on $\alpha$, so that
\begin{equation*}
\begin{split}
C>C\lambda_n\geq\lambda_n E_n^{\alpha}(\theta_n)\geq& |p_\alpha|\sum_{i=0}^{n-1}\lambda_n|\theta_n^{i}|-(\alpha+1)n\lambda_n+\lambda_n-n\lambda_n+\frac{\alpha^2}{8}+1\\
\geq&|p_\alpha|\sum_{i=0}^{n-1}\lambda_n|\theta_n^{i}|-\alpha-1,
\end{split}
\end{equation*}
from which we deduce that
\begin{equation}
\int_0^1|\theta_n(t)|\,dt<+\infty.
\label{bound}
\end{equation}
From the equiboundedness assumption (\ref{boundedness2}) there exists a constant $C>0$ such that
\begin{equation}
\sup_{n}\sum_{i=0}^{n-1}\mathcal{E}_n^i(\theta_n)\leq C<+\infty,
\label{estimate2}
\end{equation}
where we have set
\begin{equation}
\begin{split}
\mathcal{E}_n^i(\theta_n)&=\cos(\theta_n^{i}+\theta_n^{i+1})-\frac{\alpha}{2}(\cos\theta_n^{i}+\cos\theta_n^{i+1})-m_\alpha.
\end{split}
\label{simple}
\end{equation}
Now, if for every fixed $\eta>0$ we define 
\begin{equation*}
I_n(\eta):=\{i\in\{0,1,\dots,n-1\}:\,\mathcal{E}_n^i(\theta_n)>\eta\}, 
\end{equation*}
then (\ref{estimate2}) implies the existence of a uniform constant $C(\eta)$ such that
\begin{equation}
\sup_n\#I_n(\eta)\leq C(\eta)<+\infty.
\end{equation}
Let $i\in\{0,1,\dots,n-1\}$ be such that $i\not\in I_n(\eta)$; that is,
\begin{equation*}
\mathcal{E}_n^i(\theta_n)=\cos(\theta_n^{i}+\theta_n^{i+1})-\frac{\alpha}{2}(\cos\theta_n^{i}+\cos\theta_n^{i+1})-m_\alpha\leq\eta.
\label{stim1}
\end{equation*}
Let $\sigma=\sigma(\eta)>0$ be defined such that if
\begin{equation*}
0\leq\cos(\theta_1+\theta_2)-\frac{\alpha}{2}(\cos \theta_1+\cos \theta_2)-m_\alpha\leq\eta,\quad \theta_1+\theta_2=2\theta,\,\theta\in\{\pm\theta_\alpha\},
\end{equation*}
then
\begin{equation*}
|\theta_1-\theta|+|\theta_2-\theta|\leq \sigma(\eta).
\end{equation*}
As a consequence, if $i\not\in I_n(\eta)$ we deduce the existence of $\theta\in\{\pm\theta_\alpha\}$ such that
\begin{equation*}
|\theta_n^{i}-\theta|\leq\sigma\text{ and }|\theta_n^{i+1}-\theta|\leq\sigma.
\end{equation*}

Hence, up to a finite number of indices $i$, both $\theta_n^i$ and its nearest neighbour $\theta_n^{i+1}$ are close to the same minimal angle $\pm\theta_\alpha$. Namely, there exists a finite number of indices $0=i_0<i_1<\dots<i_{N_n}=n-1$ such that for all $k=1,2,\dots,N_n$ we can find $\theta_{k,n}\in\{\pm\theta_\alpha\}$ satisfying for all $i\in\{i_{k-1}+1,i_{k-1}+2,\dots,i_{k}-1\}$ the aforementioned closeness property
\begin{equation}
|\theta_n^{i}-\theta_{k,n}|\leq\sigma\,\text{ and }\,|\theta_n^{i+1}-\theta_{k,n}|\leq\sigma.
\label{3.8bis}
\end{equation}
Now, let $\{i_{j_r}\}$, $r=1,\dots, M_n$ be the maximal subset of $0=i_0<i_1<\dots<i_{N_n}=n-1$ defined by the requirement that if $\theta_{j_r,n}=\pm\theta_\alpha$ then $\theta_{j_{r+1},n}=\mp\theta_\alpha$; this means that $\{i_{j_1},\dots,i_{j_{M_n}}\}\subseteq I_n(\eta)$.  Hence, there exists $C(\eta)>0$ such that $\sum_{i=0}^{n-1}\mathcal{E}^i_n(\theta_n)\geq C(\eta)M_n$ and then $E_n^{\alpha}(\theta_n)\geq C(\eta)M_n$, so that from (\ref{boundedness2}) $M_n$ are equibounded. Thus, up to further subsequences, we can assume that $M_n=M$ and that for every $r=1,\dots,M$, $t_{i_{j_r}}^n=\lambda_ni_{j_r}\to t_r$ for some $t_r\in[0,1]$ and $\theta_{j_r,n}=\theta_{r}$. Set $S=\{t_1,\dots t_M\}$ and, for fixed $\delta>0$, $S_\delta=\bigcup_r(t_r-\delta,t_r+\delta)$. Then, by identifying $\theta_n$ with its piecewise constant interpolation, from (\ref{3.8bis}) and for $n$ large enough we get
\begin{equation*}
\sup_{t\in(0,1)\backslash S_\delta}|\theta_n(t)-\theta_r|\leq\sigma.
\end{equation*}
The previous estimates, together with (\ref{bound}) ensure that $\{\theta_n\}$ is an equicontinuous and equibounded sequence in $(0,1)\backslash S_\delta$. Thus, thanks to the arbitrariness of $\delta$, up to passing to a further subsequence (not relabelled), $\theta_n$ converges in $L^{\infty}_{loc}((0,1)\backslash S)$ (and in $L^{1}_{loc}((0,1)\backslash S)$) to a function $\theta$ such that $\theta(t)\in\{\pm\theta_\alpha\}$ for a.e. $t\in(0,1)$. Moreover, $S(\theta)\subseteq S$. Finally, by the periodicity assumption (\ref{periodic}), we have $\theta_n^0=\theta_n^n$ from which passing to the limit as $n\to\infty$ we conclude that $\theta(0+)=\theta(1+)$.\\
(ii) The proof of (ii) requires minor changes in the argument above, so we will omit it.
\end{proof}

Now we can state and prove the $\Gamma\hbox{-}$convergence result.

\begin{thm}
{\rm(i)} Let $\alpha\in[0,4)$. Then $E_n^{\alpha}$  $\Gamma\hbox{-}$converges with respect to the $L^1_{loc}\hbox{-}$topology to
\begin{equation}
E^{\alpha}(\theta)=
\begin{cases}
C_\alpha\#(S(\theta)\cap[0,1)), & \text{if }\theta\in PC_{loc}(\R),\,\theta\in\{\pm\theta_\alpha\}\\
& \text{$\theta$ is 1-periodic}\\
+\infty, & \text{otherwise}
\end{cases}
\end{equation}
on $L^{1}_{loc}(\R)$, where $C_\alpha=C(-\theta_\alpha,\theta_\alpha)$ is given by {\rm(\ref{crease})} and $PC_{loc}(\R)$ denotes the space of locally piecewise constant functions on $\R$.\\
{\rm(ii)} Let $\alpha\in[4,+\infty)$. Then $E_n^{\alpha}$  $\Gamma\hbox{-}$converges with respect to the $L^1_{loc}\hbox{-}$topology to
\begin{equation}
E^{\alpha}(\theta)=
\begin{cases}
0, & \text{if }\theta=0\\
+\infty, & \text{otherwise}
\end{cases}
\end{equation}
on $L^{1}_{loc}(\R)$.
\label{gammalim}
\end{thm}

\begin{proof}
(i) {\bf Liminf inequality.} We may assume, without loss of generality, that $\theta$ is left-continuous at each jump. Let $\theta_n\to\theta$ in $L^1(0,1)$ be such that $E_n^{\alpha}(\theta_n)<+\infty$. Then, from Proposition~\ref{compact2} there exist $N\in\N$, $\bar{\theta}_1,\dots,\bar{\theta}_N\in\{\pm\theta_\alpha\}$ and $0=s_0<s_1<\dots<s_N=1$, $\{s_j\}=\{t_k\}$ (the set of indices may be different if $s_k=s_{k+1}$ for some $k$) such that
\begin{equation}
\theta^n_{j_k}\to\bar{\theta}_j, \quad \text{on the interval }(s_{j-1},s_j),\, j\in\{1,\dots,N\}.
\label{3.11bis}
\end{equation}
For $l\in\{0,1,\dots,N\}$, let $\{k_n^l\}_n$ be a sequence of indices such that $k_n^0=0$,
\begin{equation*}
\lim_n\lambda_nk_n^l={s_l},
\end{equation*}
and let $\{h_n^l\}_n$ be another sequence of indices such that $h_n^0=0$,
\begin{equation*}
\lim_n\lambda_nh_n^l=\frac{s_l+s_{l-1}}{2}.
\end{equation*}
To get the $\Gamma$\hbox{-}liminf inequality, we rewrite the energy as follows:
\begin{equation}
E_n^\alpha(\theta_n)=\sum_{j=1}^{N-1}E_n^\alpha(\theta_n, h_n^j, h_n^{j+1})+r_n,
\label{otherenergy}
\end{equation}
where we have set
\begin{equation*}
E_n^\alpha(\theta_n, h_n^j, h_n^{j+1})=\sum_{i=h_n^j}^{h_n^{j+1}-1}\Bigl[\cos(\theta_n^{i}+\theta_n^{i+1})-\frac{\alpha}{2}(\cos\theta_n^{i}+\cos\theta_n^{i+1})-m_\alpha\Bigr],
\end{equation*}
$m_\alpha=-\frac{\alpha^2}{8}-1$ and
\begin{equation*}
r_n=\sum_{i=0}^{h_n^1-1}\mathcal{E}^i_n(\theta_n)+\sum_{i=h^{N}_n+1}^{n-1}\mathcal{E}^i_n(\theta_n),
\end{equation*}
with $r_n>0$ and $\mathcal{E}^i_n(\theta_n)$ as in (\ref{simple}). Defining for $j\in\{1,2,\dots,N-1\}$
\begin{equation}
\tilde{\theta}_n^i=
\begin{cases}
\bar{\theta}_j, & \text{if }i\leq h_n^j-k_n^j-1,\\
\theta_n^{i+k_n^j}, & \text{if }h_n^j-k_n^j\leq i\leq h_n^{j+1}-k_n^j-1,\\
\bar{\theta}_{j+1}, & \text{if }i\geq h_n^{j+1}-k_n^j,
\end{cases}
\end{equation}
we have that $\tilde{\theta}_n^i$ is a test function for the minimum problem defining $C(\theta(s_j-),\theta(s_j+))$ as in (\ref{crease}), where $\theta(s_j-)=\bar{\theta}_j$ and $\theta(s_j+)=\bar{\theta}_{j+1}$. 

For $n$ large enough and any $\sigma>0$, we then find that each summand in (\ref{otherenergy}) can be estimated from below as
\begin{equation}
\begin{split}
E_n^\alpha(\theta_n, k_n^j, k_n^{j+1})&=\sum_{i=h_n^j}^{k_n^j}\mathcal{E}^i_n(\theta_n)+\sum_{i=k^{j}_n}^{h_n^{j+1}-1}\mathcal{E}^i_n(\theta_n)\\
&=\sum_{l=h_n^j-k_n^j}^{0}\mathcal{E}^{l+k_n^j}_n(\theta_n)+\sum_{l=1}^{h_n^{j+1}-k_n^j-1}\mathcal{E}^{l+k_n^j}_n(\theta_n)\\
&=\sum_{l=h_n^j-k_n^j}^{h_n^{j+1}-k_n^j-1}\mathcal{E}^{l}_n(\tilde{\theta}_n)\\
&=\sum_{l\in\Z}\mathcal{E}^{l}_n(\tilde{\theta}_n)-\sum_{l\leq h_n^j-k_n^j-1}\mathcal{E}^{l}_n(\tilde{\theta}_n)-\sum_{l\geq h_n^{j+1}-k_n^j}\mathcal{E}^{l}_n(\tilde{\theta}_n)\\
&=\sum_{l\in\Z}\mathcal{E}^{l}_n(\tilde{\theta}_n)-\mathcal{E}^{h_n^j - k_n^j - 2}_n(\tilde{\theta}_n)\\
&\geq\sum_{i\in\Z}\Bigl[\cos(\tilde{\theta}_{n}^{i}+\tilde{\theta}_{n}^{i+1})-\frac{\alpha}{2}(\cos\tilde{\theta}_{n}^{i}+\cos\tilde{\theta}_{n}^{i+1})-m_\alpha\Bigr]-\omega(\sigma)\\
&\geq C(\bar{\theta}_j,\bar{\theta}_{j+1})-\omega(\sigma),
\end{split}
\label{stimaliminf}
\end{equation}
where $\omega:[0,+\infty)\to[0,+\infty)$ is a suitable continuous function, $\omega(0)=0$. Finally, since for every $j\in\{1,\dots,N-1\}$ it holds $C(\bar{\theta}_j,\bar{\theta}_{j+1})=C_\alpha$ by (\ref{symmetry}), combining (\ref{stimaliminf}) with (\ref{otherenergy}) and passing to the liminf as $n\to+\infty$ we get the liminf inequality 
\begin{equation}
\begin{split}
\mathop{\lim\inf}_{n\to\infty}E_n^\alpha(\theta_n)&\geq\mathop{\lim\inf}_{n\to\infty}\Bigl[(N-1)C_\alpha-(N-1)\omega(\sigma)\Bigr]\\
&=C_\alpha\#(S(\theta)\cap[0,1)),
\end{split}
\label{inliminf}
\end{equation}
where the latter equality follows by the arbitrariness of $\sigma$.
\\
{\bf Limsup inequality.} Let $\theta$ be such that $E^{\alpha}(\theta)<+\infty$. Then there exist $M\in\N$, $\bar{\theta}_1,\dots,\bar{\theta}_M\in\{\pm\theta_\alpha\}$ and $0=t_0<t_1<\dots<t_M=1$ such that $\#S(\theta)=M-1$ and
\begin{equation}
\theta(t)=\bar{\theta}_j, \quad t\in(t_{j-1},t_{j}),\,j\in\{1,2,\dots,M\}.
\label{3.19bis}
\end{equation}
Fixed $\eta>0$, from the definition of $C(\theta(t_j-),\theta(t_j+))$ for $j\in\{1,2,\dots,M-1\}$ we can find functions $\psi_{j,j+1}:\Z\to [-\pi/2,\pi/2]$, such that
\begin{equation}
\psi_{j,j+1}^i=
\begin{cases}
\bar{\theta}_j & \text{for }i\leq-N_j,\\
\bar{\theta}_{j+1} & \text{for }i\geq N_j,
\end{cases}
\label{testpsi}
\end{equation}
and 
\begin{equation}
\begin{split}
&\sum_{i\in\Z}\Bigl[\cos(\psi_{j,j+1}^{i}+\psi_{j,j+1}^{i+1})-\frac{\alpha}{2}(\cos\psi_{j,j+1}^{i}+\cos\psi_{j,j+1}^{i+1})-m_\alpha\Bigr]\\
&\leq C(\theta(t_j-),\theta(t_j+))+\eta=C_\alpha+\eta,
\end{split}
\label{contribution}
\end{equation}
where the latter equality follows again by (\ref{symmetry}).
Note that in (\ref{testpsi}) we may assume $N=N_j$ independent of $j$, up to choose $N=\displaystyle\max_{1\leq j\leq M-1}\{N_j\}$.

We define a recovery sequence $\tilde{\theta}_n$ by means of a translation argument involving functions $\psi_{j,j+1}$, $j\in\{1,\dots,M-1\}$, that will allow us to estimate the energy contribution from above with (\ref{contribution}) in a suitable neighbourhood of each jump point $t_j$. Namely, for every $j$, we set $\tilde{\theta}_n^i=\psi_{j,j+1}^{i-\lfloor t_jn\rfloor}$ if $i\in\left\{\lfloor t_jn\rfloor-N,\dots,\lfloor t_jn\rfloor+N\right\}$, while if $i\in \{\lfloor t_jn\rfloor+N,\dots,\lfloor t_{j+1}n\rfloor-N\}$, we define $\tilde{\theta}_n^i$ to be constantly equal to $\bar{\theta}_{j+1}$, according to (\ref{testpsi}). 
This definition can be summarized as follows:
\begin{equation}
\tilde{\theta}_n^i=
\begin{cases}
\bar{\theta}_1 & \text{if }0\leq i\leq \lfloor t_1n\rfloor-N\\
\psi_{j,j+1}^{i-\lfloor t_jn\rfloor}& \text{if }\lfloor t_jn\rfloor-N\leq i\leq \lfloor t_{j}n\rfloor+N,\\
\bar{\theta}_{j+1} & \text{if }\lfloor t_jn\rfloor+N\leq i\leq\lfloor t_{j+1}n\rfloor-N, \quad j\in\{1,\dots,M-1\}\\
\bar{\theta}_M & \text{if }n-N\leq i \leq n-1.
\end{cases}
\end{equation}
We note that the corresponding $\tilde{\theta}_n\in\mathcal{D}_n(I)$ satisfy $\tilde{\theta}_n\to\theta$ in $L^\infty$ and (here we use the simplified notation for the energies as in (\ref{simple}))
\begin{equation*}
\begin{split}
E_n^\alpha(\tilde{\theta}_n)=\sum_{i=0}^{n-1}\mathcal{E}_n^i(\tilde{\theta}_n)&=\sum_{i=0}^{\lfloor t_1n\rfloor-N-1}\mathcal{E}_n^i(\tilde{\theta}_n)+\sum_{j=1}^{M-1}\left(\sum_{i=\lfloor t_jn\rfloor-N}^{\lfloor t_jn\rfloor+N-1}\mathcal{E}_n^i(\tilde{\theta}_n)\right)\\
&+\sum_{j=1}^{M-1}\left(\sum_{i=\lfloor t_jn\rfloor+N}^{\lfloor t_{j+1}n\rfloor-N-1}\mathcal{E}_n^i(\tilde{\theta}_n)\right)+\sum_{i=n-N}^{n-1}\mathcal{E}_n^i(\tilde{\theta}_n)\\
&=\sum_{j=1}^{M-1}\left(\sum_{i=\lfloor t_jn\rfloor-N}^{\lfloor t_jn\rfloor+N-1}\mathcal{E}_n^i(\psi_{j,j+1}^{i-\lfloor t_jn\rfloor})\right)=\sum_{j=1}^{M-1}\sum_{i\in\Z}\mathcal{E}_n^i(\psi_{j,j+1}^i)\\
&\leq (M-1)(C_\alpha+\eta),
\end{split}
\end{equation*}
whence, by the arbitrariness of $\eta$, we deduce that
\begin{equation}
\mathop{\lim\sup}_{n\to+\infty}E_n^\alpha(\tilde{\theta}_n)\leq (M-1)(C_\alpha+\eta)=C_\alpha\#(S(\theta)\cap[0,1)).
\label{inlimsup}
\end{equation}
Thus, (\ref{inlimsup}) shows that the lower bound (\ref{inliminf}) is sharp, and this concludes the proof of (i).\\
(ii) In this case the proof is immediate. Indeed, for any $\theta_n\to0$, from $E_n^\alpha(\theta_n)\geq0$ we have in particular that 
\begin{equation*}
\mathop{\lim\inf}_{n\to+\infty}E_n^\alpha(\theta_n)\geq0.
\end{equation*}
As a recovery sequence, we can choose $\theta_n\equiv0$, for which we obtain $\displaystyle\lim_{n\to+\infty}E_n^\alpha(\theta_n)=0$.
\end{proof}

\section{Limit behaviour near the transition point $\alpha=4$}

The description of the limit as $n\to+\infty$ of the energies $E_n^\alpha$ with fixed $\alpha$, carried out in the previous section, has a gap for $\alpha=4$. Indeed, the crease energy $C_\alpha$ jumps from a strictly positive value (corresponding to $0\leq\alpha<4$) to $0$ (when $\alpha\geq4$). Note also that the explicit value of $C_\alpha$ defined implicitly by (\ref{crease}) is not known in literature. This suggests to focus near the transition point $\alpha=4$, let the parameter $\alpha$ depend on $n$ and be close to $4$ from below; that is, replace $\alpha$ by $4-\alpha_n$, $\alpha_n\to4^-$. 

Such analysis is the main content of a recent paper by Cicalese and Solombrino~\cite{CS}, where they find suitable scaling and order parameter to compute the energy the system spends in a transition between two states with different chirality when $\alpha\simeq4$. Moreover, they show the dependence of the limit on the particular sequence $\alpha_n\to4^-$ and the existence of different regimes.
Our aim is to show that their result can be retrieved also correspondingly to a different choice of the order parameter in the energies. 
First of all, we write the energies (\ref{exten}) in terms of $4-\alpha$ as
\begin{equation}
E_n^{\alpha}(\theta)=[4-(4-\alpha)]\sum_{i=0}^{n-1}(1-\cos\theta^i)-\sum_{i=0}^{n-1}[1-\cos(\theta^i+\theta^{i+1})]+n\frac{(4-\alpha)^2}{8},
\label{energy3}
\end{equation}
and when necessary, we may think also the quantities $W_\alpha$, $C_\alpha$, etc. to be functions of $4-\alpha$. 
Note that if we choose as a test function in (\ref{crease}) $\theta^{i,\alpha}=\text{sign}(i)\arccos(\alpha/4)$ then we obtain a first rough estimate
\begin{equation*}
0<C_{\alpha}\leq (4-\alpha)-\frac{(4-\alpha)^2}{8},
\end{equation*}
showing in particular that $C_{\alpha}\to0$ as $\alpha\to4^-$.

The following proposition (compare with \cite[Proposition~4.3]{CS}) characterizes the angles between neighbours for an equibounded (in energy) sequence of spins as the frustration parameter approaches the critical value from below.

\begin{prop}
If $\{\theta_n\}$ is a sequence such that

\begin{equation}
\sup_nE_n^{\alpha_n}(\theta_n)\leq C(4-\alpha_n)^{3/2},
\end{equation}
\\
then $\theta_n^i\to0$ as $\alpha_n\to4^-$ uniformly with respect to $i\in\{0,1,\dots,n-1\}$.
\label{compact}
\end{prop}

\begin{proof}
The claim follows immediately from the estimate
\begin{equation}
0\leq2\Bigl(\cos\theta_n^i-\frac{\alpha_n}{4}\Bigr)^2\leq\sum_{i=0}^{n-1}W_{\alpha_n}(\theta_n^i)\leq E_n^{\alpha_n}(\theta_n)\leq C(4-\alpha_n)^{3/2}
\end{equation}
valid for all $i\in\{0,1,\dots,n-1\}$.
\end{proof}

We introduce a new order parameter
\begin{equation}
v^i_n=\frac{\theta^i_n}{\theta_{\alpha_{n}}}
\label{newparameter}
\end{equation}
and reformulate the $\Gamma\hbox{-}$convergence result by Cicalese and Solombrino (\cite[Theorem~4.2]{CS}) in terms of this new variable. 
However, it is worth noting that the ``flat'' angular parameter $v_n^i$ is equivalent with their variable $z^i_n$ in the regime of ``small angles'', i.e., as $\alpha_n\to4^-$, $\theta^i_n\to0$, since in this case
\begin{equation*}
\theta_{\alpha_n}=\arccos\left(1-\frac{(4-\alpha_n)}{4}\right)\simeq\frac{\sqrt{4-\alpha_n}}{\sqrt{2}}
\end{equation*}
and
\begin{equation*}
z^i_n=\frac{2\sqrt{2}}{\sqrt{4-\alpha_n}}\sin\left(\frac{\theta^i_n}{2}\right)\simeq\frac{\sqrt{2}\theta^i_n}{\sqrt{4-\alpha_n}}.
\end{equation*}

The change of variables (\ref{newparameter}) associates to any given $\theta_n\in \mathcal{D}_n(I)$ a piecewise-constant function $v_n\in\widetilde{\mathcal{D}}_n(I)$ where
\begin{equation*}
\widetilde{\mathcal{D}}_n(I):=\Bigl\{v:[0,1)\to\R:\, v(t)=v^i_n\text{ if }t\in\lambda_n(i+[0,1)),\, i\in\{0,1,\dots,n-1\}\Bigr\},
\end{equation*}
with $v_n$ as in (\ref{newparameter}). With a slight abuse of notation, we regard $E_n^{\alpha_n}$ as a functional defined on $v\in L^1(I,\R)$ by
\begin{equation}
E_n^{\alpha_n}(v)=
\begin{cases}
E_n^{\alpha_n}(\theta), & \text{if }v\in\widetilde{\mathcal{D}}_n(I)\\
+\infty, &\text{otherwise},
\end{cases}
\label{abusenergy}
\end{equation}
and correspondingly we define the scaled energies
\begin{equation}
F_n^{\alpha_n}(v):=\frac{8E_n^{\alpha_n}(v)}{\sqrt{2}{(4-\alpha_n)^{3/2}}}.
\label{abusenergy2}
\end{equation}

\begin{thm}[Cicalese and Solombrino~\cite{CS}]
Let $F_n^{\alpha_n}: L^1(I,\R)\to[0,+\infty]$ be the functional in {\rm(\ref{abusenergy2})}. Assume that there exists $l:=\lim_n \sqrt{2}\lambda_n/4(4-\alpha_n)^{1/2}$. Then $F^0(v):=\displaystyle\Gamma\hbox{-}\lim_nF_n^{\alpha_n}(v)$ with respect to the $L^1(I)$ convergence is given by:\\
\\
{\rm (i)} if $l=0$,
\begin{equation}
F^0(v):=
\begin{cases}
\frac{8}{3}\#(S(v)) & \text{if }v\in BV(I,\{\pm1\}),\\
+\infty & \text{otherwise.}
\end{cases}
\end{equation}
{\rm(ii)} if $l\in(0,+\infty)$,
\begin{equation}
F^0(v):=
\begin{cases}
\displaystyle\frac{1}{l}\int_{I}\Bigl(v^2(t)-1\Bigr)^2\,dt + {l}\int_{I}(\dot{v}(t))^2\,dt& \text{if }v\in W^{1,2}_{|per|}(I),\\
+\infty & \text{otherwise.}
\end{cases}
\end{equation}
where we have set $W^{1,2}_{|per|}(I):=\{v\in W^{1,2}(I):\, |v(0)|=|v(1)|\}$.\\
\\
{\rm(iii)} if $l=+\infty$,
\begin{equation}
F^0(v):=
\begin{cases}
0 & \text{if }v=const.,\\
+\infty & \text{otherwise.}
\end{cases}
\end{equation}
\label{cicasol}
\end{thm}

\begin{proof}

In order to simplify the notation, we put $\varepsilon_n:=4-\alpha_n\to0$ as $n\to\infty$. Let $\{v_n\}$ be a sequence in $\widetilde{\mathcal D}_n(I)$ such that $\displaystyle\sup_n\frac{E_n^{\varepsilon_n}(v_n)}{\varepsilon_n^{3/2}}\leq C<\infty$. As remarked before, correspondingly, there exists a sequence $\{\theta_n\}$ in $\mathcal{D}_n(I)$ such that $\displaystyle\sup_n\frac{E_n^{\varepsilon_n}(\theta_n)}{\varepsilon_n^{3/2}}\leq C<\infty$, satisfying $\theta^i_n\to0$ uniformly with respect to $i$ by Proposition~\ref{compact}. 

From the estimates contained in the proof of \cite[Theorem~4.2]{CS} we get

\begin{equation*}
E_n^{\varepsilon_n}(\theta_n)\geq 8\sum_{i=0}^{n-1}\Bigl[\sin^2\Bigl(\frac{\theta_n^i}{2}\Bigr)-\frac{\varepsilon_n}{8}\Bigr]^2+2(1-\gamma_n)\sum_{i=0}^{n-1}\Bigl[\sin\Bigl(\frac{\theta_n^{i+1}}{2}\Bigr)-\sin\Bigl(\frac{\theta_n^{i}}{2}\Bigr)\Bigr]^2,
\end{equation*}
for some $\gamma_n\to0$. Since $\sin\theta\simeq\theta$ as $\theta\to0$, we may improve the estimate obtaining
\begin{equation*}
E_n^{\varepsilon_n}(\theta_n)\geq 8(1-\gamma'_n)\sum_{i=0}^{n-1}\Bigl[\Bigl(\frac{\theta_n^i}{2}\Bigr)^2-\frac{\varepsilon_n}{8}\Bigr]^2+\frac{(1-\gamma''_n)}{2}\sum_{i=0}^{n-1}\Bigl[\Bigl(\frac{\theta_n^{i+1}}{2}\Bigr)-\Bigl(\frac{\theta_n^{i}}{2}\Bigr)\Bigr]^2,
\end{equation*}
for suitable $\gamma'_n,\gamma''_n\to0$. In terms of the new order parameter $v^i_n$ defined by (\ref{newparameter}) the previous inequality now reads
\begin{equation*}
\begin{split}
E_n^{\varepsilon_n}(\theta_n)&\geq \frac{2{\theta^4_{\varepsilon_n}}}{\lambda_n}(1-\gamma'_n)\sum_{i=0}^{n-1}\lambda_n\Bigl[(v_n^i)^2-\frac{\varepsilon_n}{2{\theta^2_{\varepsilon_n}}}\Bigr]^2\\
&+\frac{\theta^2_{\varepsilon_n}\lambda_n}{8}(1-\gamma''_n)\sum_{i=0}^{n-1}\lambda_n\Bigl(\frac{v_n^{i+1}-v_n^{i}}{\lambda_n}\Bigr)^2,
\end{split}
\end{equation*}
where $\lambda_n=\frac{1}{n}$. If we multiply both the sides by $8/\sqrt{2}{\varepsilon^{3/2}_n}$, since $\displaystyle\frac{\varepsilon_n}{2{\theta^2_{\varepsilon_n}}}\to1$ we get
\begin{equation*}
\begin{split}
\frac{8E_n^{\varepsilon_n}(\theta_n)}{\sqrt{2}{\varepsilon^{3/2}_n}}&\geq \frac{8\sqrt{2}{\theta^4_{\varepsilon_n}}}{\lambda_n{\varepsilon^{3/2}_n}}(1-\gamma'_n)\sum_{i=0}^{n-1}\lambda_n\Bigl[(v_n^i)^2-1\Bigr]^2\\
&+\frac{\sqrt{2}\theta^2_{\varepsilon_n}\lambda_n}{2{\varepsilon^{3/2}_n}}(1-\gamma''_n)\sum_{i=0}^{n-1}\lambda_n\Bigl(\frac{v_n^{i+1}-v_n^{i}}{\lambda_n}\Bigr)^2.
\end{split}
\end{equation*}
Since $\theta_{\varepsilon_n}=\arccos(1-\frac{\varepsilon_n}{4})$ and $\theta_{\varepsilon_n}\simeq\frac{\sqrt{\varepsilon_n}}{\sqrt{2}}$ as $\varepsilon_n\to0$, we note that $\displaystyle\frac{8\sqrt{2}{\theta^4_{\varepsilon_n}}}{\lambda_n{\varepsilon^{3/2}_n}}\simeq \frac{4\sqrt{\varepsilon_n}}{\sqrt{2}\lambda_n}$ and $\displaystyle\frac{\sqrt{2}\theta^2_{\varepsilon_n}\lambda_n}{2{\varepsilon^{3/2}_n}}\simeq\frac{\sqrt{2}\lambda_n}{4\sqrt{\varepsilon_n}}$ as $n\to\infty$. Thus, we finally get
\begin{equation}
\begin{split}
\frac{8E_n^{\varepsilon_n}(\theta_n)}{\sqrt{2}{\varepsilon^{3/2}_n}}&\geq \frac{4\sqrt{\varepsilon_n}}{\sqrt{2}\lambda_n}(1-\widetilde{\gamma}'_n)\sum_{i=0}^{n-1}\lambda_n\Bigl[(v_n^i)^2-1\Bigr]^2\\
&+\frac{\sqrt{2}\lambda_n}{4\sqrt{\varepsilon_n}}(1-\widetilde{\gamma}''_n)\sum_{i=0}^{n-1}\lambda_n\Bigl(\frac{v_n^{i+1}-v_n^{i}}{\lambda_n}\Bigr)^2,
\label{finalestimate}
\end{split}
\end{equation}
for suitable $\widetilde{\gamma}'_n, \widetilde{\gamma}''_n\to0$. The estimate (\ref{finalestimate}) implies the liminf inequality both in case (i) and (ii) as remarked in \cite{CS}, and the limsup inequality can be obtained in both cases by the constructive argument contained therein, so we will omit the proof.
\end{proof}

\begin{oss}\label{osservazione} (asymptotic behaviour of $C_\alpha$). As remarked before, $C_{\alpha}\to0$ as $\alpha\to4^-$. However, we may use Theorem~\ref{cicasol}(i) to refine this estimate and determine the right order of $C_{\alpha_n}$ with respect to $4-\alpha_n$ as $\alpha_n\to4$.

In the regime $\lambda_n<\!<\!(4-\alpha_n)^{1/2}$ we can compute the limit of energies $F_n^{\alpha_n}(v)$ first as $n\to\infty$ while keeping $\alpha_n\equiv\alpha_0\neq4$ fixed, and then the limit as $\alpha_0\to4^-$. Thanks to Theorem~\ref{gammalim} and the continuity result ensured by Proposition~\ref{continuity}, we obtain
\begin{equation}
F^{\alpha_0}(v):=\displaystyle\Gamma\hbox{-}{\lim_{n\to+\infty}}\frac{8E_n^{\alpha_n}(v)}{\sqrt{2}{(4-\alpha_n)^{3/2}}}=\frac{8C_{\alpha_0}}{\sqrt{2}{(4-\alpha_0)^{3/2}}}\#(S(v)),
\label{consistent}
\end{equation}
whence, by means of Theorem~\ref{cicasol}(i), we get
\begin{equation}
F^0(v):=\displaystyle\Gamma\hbox{-}{\lim_{\alpha_0\to4}}F^{\alpha_0}(v)=\frac{8}{3}\#(S(v)). 
\end{equation}
The convergence of minimum problems as $\alpha\to4^-$ finally gives
\begin{equation}
\lim_{\alpha\to4^-}\frac{3C_\alpha}{\sqrt{2}(4-\alpha)^{3/2}}=1.
\label{stimarussi}
\end{equation}
\end{oss}

Thus, near the ferromagnet-helimagnet transition point, the energy $C_\alpha$ coincides with the energy $E_{dw}\propto (4-\alpha)^{3/2}$ for the excitation of a chiral domain wall separating two domains of opposite chirality, which is a well-known universal low-temperature property of frustrated classical spin chains (see, e.g., Dmitriev and Krivnov~\cite{DK}).

\section{A link with the gradient theory of phase transitions}

In this section we show that the variational asymptotic behaviour of the energies $F_n^\alpha$ for any $\alpha\in[0,4]$, both in the case of fixed $\alpha$ (Theorem~\ref{gammalim}) and in the case $\alpha\simeq4$ (Theorem~\ref{cicasol}), is the same as that of a parametrized family of Modica-Mortola type functionals, thus providing an interesting connection between frustrated lattice spin systems and the gradient theory of phase transitions (see also \cite[Section~6]{BC}). 

In order to do that in the framework of the \emph{equivalence by $\Gamma\hbox{-}$convergence}, we recall some definitions about $\Gamma\hbox{-}$equivalence for families of parametrized functionals, uniform equivalence, regular and singular points, as introduced by Braides and Truskinovsky~\cite{BT}.

\begin{defn}[$\Gamma\hbox{-}$equivalence]\label{defequiv} Let $\mathcal{A}$ be a set of parameters. Two families of parametrized functionals $F_n^\alpha$ and $G_n^\alpha$ are \emph{equivalent at scale 1 at $\alpha_0\in\mathcal{A}$} if $F_n^{\alpha_0}$ and $G_n^{\alpha_0}$ are equivalent at scale 1, i.e.,
\begin{equation}
\displaystyle\Gamma\hbox{-}{\lim_{n\to+\infty}}F_n^{\alpha_0}=\displaystyle\Gamma\hbox{-}{\lim_{n\to+\infty}}G_n^{\alpha_0}
\end{equation}
and these $\Gamma\hbox{-}$limits are non-trivial.
\end{defn}

\begin{defn}[uniform $\Gamma\hbox{-}$equivalence] Let $\mathcal{A}$ be a set of parameters. Two families of parametrized functionals $F_n^\alpha$ and $G_n^\alpha$ are \emph{uniformly equivalent at scale 1 at $\alpha_0\in{\mathcal{A}}$} if for all 
$\alpha_n\to\alpha_0$ we have, up to subsequences,
\begin{equation}
\displaystyle\Gamma\hbox{-}{\lim_{n\to+\infty}}F_n^{\alpha_n}=\displaystyle\Gamma\hbox{-}{\lim_{n\to+\infty}}G_n^{\alpha_n}
\end{equation}
and these $\Gamma\hbox{-}$limits are non-trivial. They are uniformly equivalent on $\mathcal{A}$ if they are uniformly equivalent at all $\alpha_0\in\mathcal{A}$.
\end{defn}

\begin{defn}[regular point] $\alpha_0\in\mathcal{A}$ is a \emph{regular point} if for all 
$\alpha_n\to\alpha_0$ we have, up to a subsequence,
\begin{equation}
\displaystyle\Gamma\hbox{-}{\lim_{n\to+\infty}}F_n^{\alpha_n}=\displaystyle\Gamma\hbox{-}{\lim_{n\to+\infty}}F_n^{\alpha_0}.
\end{equation}
\end{defn}

\begin{defn}[singular point] $\alpha_0\in\mathcal{A}$ is a \emph{singular point} if it is not regular; that is, if 
there exist $\alpha'_n\to\alpha_0$, $\alpha''_n\to\alpha_0$ such that (up to subsequences)
\begin{equation}
\displaystyle\Gamma\hbox{-}{\lim_{n\to+\infty}}F_n^{\alpha'_n}\neq\displaystyle\Gamma\hbox{-}{\lim_{n\to+\infty}}F_n^{\alpha''_n}.
\end{equation}
\end{defn}

According to the previous definitions, each $0\leq\alpha_0<4$ is a regular point for $F_n^{\alpha}$, since, as already observed in Remark~\ref{osservazione}, for any sequence $\alpha_n\to\alpha_0$, we have
\begin{equation*}
F^{\alpha_0}(v):=\displaystyle\Gamma\hbox{-}{\lim_{n\to+\infty}}F_n^{\alpha_n}(v)=\frac{8C_{\alpha_0}}{\sqrt{2}{(4-\alpha_0)^{3/2}}}\#(S(v)).
\end{equation*}
As a consequence of Theorem~\ref{cicasol}, instead, $\alpha_0=4$ is a singular point for $F_n^{\alpha}$.

\begin{figure}[htp]
\centering
\begin{tikzpicture}
\begin{axis} [axis lines=middle,
enlargelimits,
xmax=5,ymax=3,
xtick={},
xticklabels={},
yticklabels={},
xlabel=$\alpha$,ylabel=$\frac{1}{n}$]
\addplot [domain=0.5:4,
samples=40,smooth,thick,blue]
{sqrt(4-x)};
\draw [dashed]({axis cs:4,0}|-{rel axis cs:0,1}) -- ({axis cs:4,0}|-{rel axis cs:0,0});
\node at
(rel axis cs:0.45,0.25) {Helimagnetic};
\node[fill=white] at
(rel axis cs:0.45,0.65) {Ferromagnetic};
\end{axis}
\end{tikzpicture}
\caption{The $\frac{1}{n}-\alpha$ space. In blue the failure curve $\frac{1}{n}=(4-\alpha)^{1/2}$.}
\label{picture}
\end{figure}
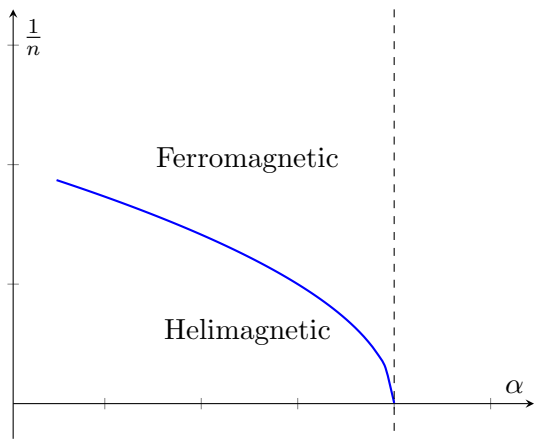

The behaviour of the system close to the transition point $\alpha=4$ can be pictured in the $\frac{1}{n}\hbox{--}\alpha$ plane (see Fig.~\ref{picture}), where the crossover line $\frac{1}{n}=(4-\alpha)^{1/2}$ separates a zone where there is \emph{helimagnetic order} ($\frac{1}{n}<\!<\!(4-\alpha)^{1/2}$) from one where we have \emph{ferromagnetic order} ($\frac{1}{n}>\!>\!(4-\alpha)^{1/2}$).

For our purposes, it is useful to recall the well known $\Gamma\hbox{-}$convergence result in gradient theory of phase transitions due to Modica and Mortola~\cite{MM}. Let $\Omega\subset\R$ be an open set, $u:\Omega\to\R$ and $W=W(u)$ a non-convex energy such that $W\geq0$, $W(u)\geq c(u^2-1)$ and $W=0$ if and only if $u=a,b$. $W$ is called a \emph{double-well} potential. Let $C>0$ and consider the energies

\begin{equation}
F_n(u)=n\int_\Omega W(u)\,dx+\frac{C^2}{n}\int_\Omega(\dot{u})^2\,dx,\quad u\in W^{1,2}(\Omega).
\end{equation}

\begin{thm}[Modica-Mortola's theorem] The functionals $F_n$ above $\Gamma$\hbox{-}\\converge as $n\to\infty$ and with respect to the $L^1(\Omega)$ convergence to the functional

\begin{equation}
F_\infty(u)=
\begin{cases}
C\cdot c_W\#(S(u)\cap\Omega), & \text{if }u\in\{a,b\} \text{ a.e.}\\
+\infty, & \text{otherwise},
\end{cases}
\end{equation}
\\
where $c_W:=2\int_a^b\sqrt{W(s)}\,ds$.
\label{ModicaMortola}
\end{thm}

The following theorem states the announced uniform equivalence by $\Gamma\hbox{-}$convergence of energies $F_n^{\alpha}$ with parametrized Modica-Mortola type functionals.

\begin{thm}
Setting $\lambda_{n,\alpha}:=2n{\theta^4_{\alpha}}$, $M_\alpha:=3C_{\alpha}/8$, $\mu_\alpha:=\frac{\sqrt{2}(4-\alpha)^{3/2}}{8}$, the energies
\begin{equation*}
G_n^{\alpha}(v)=
\begin{cases}
\displaystyle\frac{1}{\mu_\alpha}\Bigl[\lambda_{n,\alpha}\int_{I}\Bigl(v^2-1\Bigr)^2\,dt + \frac{M_\alpha^2}{\lambda_{n,\alpha}}\int_{I}(\dot{v})^2\,dt\Bigr], & \text{if $v\in W^{1,2}_{|per|}(I)$,}\\
+\infty, & \text{otherwise in $L^1_{loc}(\R)$,}
\end{cases}
\end{equation*}
and $F_n^{\alpha}(v):=\frac{1}{\mu_\alpha} E_n^{\alpha}(v)$ are uniformly equivalent by $\Gamma$\hbox{-}convergence on $[0,4]$. Moreover,
\begin{enumerate}
\item[\rm(i)] each $\alpha_0\in[0,4)$ is a regular point;
\item[\rm(ii)] $\alpha_0=4$ is a singular point.
\end{enumerate}
\label{equivalence}
\end{thm}

\begin{proof} As before, we put $\varepsilon_n:=4-\alpha_n$, so that $\varepsilon_n\geq0$.\\
(i) Let $\varepsilon_n\to\varepsilon_0\neq0$ and $\{v_n\}$ be a sequence with equibounded energy. Correspondingly, by (\ref{newparameter}) we may find a sequence $\{\theta_n\}$ such that $v_n=\theta_n/\theta_{\varepsilon_n}$. The continuity of the energies $G_n^{\varepsilon}$ with respect to the parameter $\varepsilon=\varepsilon_n$, ensured also by Proposition~\ref{continuity}, allows us to consider, without loss of generality, the energies
\begin{equation*}
G_n^{\varepsilon_0}(v)=
\displaystyle\frac{8}{\sqrt{2}\varepsilon_0^{3/2}}\Bigl[{2n{\theta^4_{\varepsilon_0}}}\int_{I}\Bigl(v^2-1\Bigr)^2\,dt + \frac{1}{2n\theta_{\varepsilon_0}^4}\Bigl(\frac{3C_{\varepsilon_0}}{8}\Bigr)^2\int_{I}(\dot{v})^2\,dt\Bigr],
\end{equation*}
where $v=\theta/\theta_{\varepsilon_0}$. After simplifying the constants, we may rewrite the energies in terms of $\theta$ as
\begin{equation*}
G_n^{\varepsilon_0}(\theta)=
\displaystyle\frac{8}{\sqrt{2}\varepsilon_0^{3/2}}\Bigl[{2n}\int_{I}\Bigl(\theta^2-\theta_{\varepsilon_0}^2\Bigr)^2\,dt + \frac{1}{2n}\Bigl(\frac{3C_{\varepsilon_0}}{8 \theta_{\varepsilon_0}^3}\Bigr)^2\int_{I}(\dot{\theta})^2\,dt\Bigr].
\end{equation*}
In order to compute the $\Gamma\hbox{-}$limit as $n\to\infty$ of the energies $G_n^{\varepsilon_0}$, we may apply the $\Gamma\hbox{-}$convergence result by Modica and Mortola (Theorem~\ref{ModicaMortola}), thus obtaining
\begin{equation*}
G^{\varepsilon_0}(\theta):=\displaystyle\Gamma\hbox{-}{\lim_{n\to+\infty}}G_n^{\varepsilon_0}(\theta)=\frac{8C_{\varepsilon_0}}{\sqrt{2}\varepsilon_0^{3/2}}\#(S(\theta)),
\end{equation*}
since
\begin{equation*}
c_{W}=2\int_{-\theta_{\varepsilon_0}}^{\theta_{\varepsilon_0}}|\theta^2-\theta_{\varepsilon_0}^2|\,d\theta=\frac{8}{3}\theta_{\varepsilon_0}^3.
\end{equation*}
This result coincides with (\ref{consistent}), once we remark that $\#(S(v))=\#(S(\theta))$.
\\
(ii) Let $\varepsilon_n\to0$ and $v\in W^{1,2}_{loc}(\R)$. The estimates contained in the proof of Theorem~\ref{cicasol} and equation (\ref{stimarussi}) allow us to rewrite the functional $G^{\varepsilon_n}_n$ as
\begin{equation*}
G_n^{\varepsilon_n}(v)=
\displaystyle\frac{4n\sqrt{\varepsilon_n}}{\sqrt{2}}(1+\eta_n)\int_{I}\Bigl(v^2-1\Bigr)^2\,dt +\frac{\sqrt{2}}{4n\sqrt{\varepsilon_n}}(1+\eta'_n)\int_{I}(\dot{v})^2\,dt,
\end{equation*}
for suitable sequences $\eta_n,\eta'_n\to0$. In order to simplify the notation, we put
\begin{equation*}
K_n:=\frac{\sqrt{2}}{4n\sqrt{\varepsilon_n}},
\end{equation*}
and we write
\begin{equation*}
G_n^{\varepsilon_n}(v)=
\displaystyle\frac{1}{K_n}(1+\eta_n)\int_{I}\Bigl(v^2-1\Bigr)^2\,dt +K_n(1+\eta'_n)\int_{I}(\dot{v})^2\,dt.
\end{equation*}
We distinguish between three cases:\\
\noindent
(a) $K_n\to0$. In this case, we apply again Theorem~\ref{ModicaMortola} (with $C=1$), thus obtaining
\begin{equation}
G^0(v):=\displaystyle\Gamma\hbox{-}{\lim_{n\to+\infty}}G_n^{\varepsilon_n}(v)=\frac{8}{3}\#(S(v)),
\end{equation}
since $c_W=2\int_{-1}^1|v^2-1|\,dv=\frac{8}{3}$.\\
\noindent
(b) $K_n\to l\in(0,+\infty)$. A sequence $v_n$ with equibounded energy is weakly compact in $W^{1,2}_{|per|}(I)$, then by lower semicontinuity in $W^{1,2}_{|per|}(I)$ we get
\begin{equation*}
\liminf_nG_n^{\varepsilon_n}(v_n)\geq\displaystyle\frac{1}{l}\int_{I}\Bigl(v^2-1\Bigr)^2\,dt +l\int_{I}(\dot{v})^2\,dt.
\end{equation*}
In order to obtain the limsup inequality, we can argue by density considering $v_n\in W^{1,2}_{|per|}(I)\cap C^\infty(\overline{I})$, $v_n\to v$ such that
\begin{equation*}
\lim_nG_n^{\varepsilon_n}(v_n)=\displaystyle\frac{1}{l}\int_{I}\Bigl(v^2-1\Bigr)^2\,dt +l\int_{I}(\dot{v})^2\,dt.
\end{equation*}
\noindent
(c) $K_n\to+\infty$. Let $v$ be a constant function, and consider the constant sequence $v_n\equiv v$. Trivially,
\begin{equation*}
\liminf_n G_n^{\varepsilon_n}(v_n)=\liminf_n\Bigl(\displaystyle\frac{1}{K_n}(1+\eta_n)\int_{I}\Bigl(v^2_n-1\Bigr)^2\,dt\Bigr)\geq0,
\end{equation*}
and
\begin{equation*}
\lim_n G_n^{\varepsilon_n}(v_n)=0.
\end{equation*}
\end{proof}

The proof of point (i) of Theorem~\ref{equivalence} permits us to deduce an equivalence result also for the energies $E_n^{\alpha}(\theta)$ defined in (\ref{exten}) with Modica Mortola type functionals whose potentials $\mathcal{W}_\alpha(\theta):=(\theta^2-{\theta^2_{\alpha}})^2$ have the wells located at the minimal angles $\theta=\pm\theta_\alpha$. It can be stated as follows.

\begin{cor}
Let $\alpha$ be a positive number, $\alpha\in[0,4)$. The energies $E_n^{\alpha}(\theta)$ and the family of functionals $H_n^{\alpha}(\theta)$ defined on $L^1_{loc}(\R)$ as
\begin{equation*}
H_n^{\alpha}(\theta)=
\begin{cases}
\displaystyle\frac{\lambda_{n,\alpha}}{\theta^4_{\alpha}}\int_{I}\Bigl(\theta^2(t)-{\theta^2_{\alpha}}\Bigr)^2\,dt + \frac{M_\alpha^2}{\lambda_{n,\alpha}\theta_{\alpha}^2}\int_{I}(\dot{\theta}(t))^2\,dt, & \text{if }\theta\in W^{1,2}_{|per|}(I),\\
+\infty, &\text{otherwise},
\end{cases}
\end{equation*}
are uniformly equivalent by $\Gamma$\hbox{-}convergence on $[0,4)$.
\label{corollario}
\end{cor}

\section*{Acknowledgements} 
We are grateful to Andrea Braides for suggesting this problem, and we would like to thank him for his advices and many fruitful discussions. We also thank Marco Cicalese, Francesco Solombrino and Leonard Kreutz for some interesting remarks leading to improve the manuscript. The first author gratefully acknowledges the hospitality of the Department of Mathematics, University of Rome ``Tor Vergata'', where a substantial part of this work has been carried out, and the financial support of PRIN~2010, project ``Discrete and continuum variational methods for solid and liquid crystals''.

{\bf This is a pre-print of an article published in Journal of Elasticity. The final authenticated version is available online at: \url{https://doi.org/10.1007/s10659-017-9668-8}}


%
%

\end{document}